\documentclass{article}

\usepackage{amsmath, amsthm, amssymb, mathrsfs}
\usepackage{graphicx, tikz}
\usepackage{hyperref} 
\usepackage{todonotes}
\usepackage{subfigure}
\usepackage[noend]{algpseudocode}

\usepackage{algorithmicx,algorithm}

\graphicspath{{Figure/}}

\makeatletter
\def\EquationsBySection{\def\theequation
	{\thesection.\arabic{equation}}%
	\@addtoreset{equation}{section}}

\makeatother \EquationsBySection

\newtheorem{example}{\noindent Example}
\newtheorem{theorem}{Theorem}[section]

\newtheorem{lemma}{Lemma}[section]

\newtheorem{remark}{Remark}[section]
\renewcommand{\vec}[1]{\boldsymbol{#1}}
\newcommand{\abs}[1]{\left\vert#1\right\vert}

\newcommand{\norm}[1]{\left\Vert#1\right\Vert}
\newcommand{\set}[1]{\left\{#1\right\}}

\renewcommand{\vec}[1]{\boldsymbol{#1}}
\allowdisplaybreaks[3]

\begin{document}
	
	\title{Fourier-Galerkin method for scattering poles of sound soft obstacles
		\thanks{The research of Y. Ma was supported  by the NSFC under grant No.11901085. The research of J. Sun was partially supported by an NSF Grant DMS-2109949 and a SIMONS Foundation Collaboration Grant 711922.
	}}
	
	\author{
		Yunyun Ma\thanks{
			School of Computer Science and Technology, Dongguan University of Technology, Dongguan 523808, P. R. China.
			{\it mayy007@foxmail.com}}\quad 
		Jiguang Sun\thanks{Department of Mathematical Sciences, Michigan Technological University, Houghton, MI 49931, U.S.A. {\it jiguangs@mtu.edu} (Corresponding author)}
	}
	
	\maketitle
	
	\begin{abstract}
		The computation of scattering poles for a sound-soft obstacle is investigated. These poles correspond to the eigenvalues of two boundary integral operators. We construct novel decompositions of these operators and show that they are Fredholm. Then a Fourier-Galerkin method is proposed for discretization. By establishing the regular convergence of the discrete operators, an error estimate is established using the abstract approximation theory for eigenvalue problems of holomorphic Fredholm operator functions. We give details of the numerical implementation. Several examples are presented to validate the theory and demonstrate the effectiveness of the proposed method. 
	\end{abstract}
	
	\noindent{\it Keywords}:  scattering poles, boundary integral operators, Fourier-Galerkin method, holomorphic Fredholm operator functions
	
	\section{Introduction}
	Scattering resonances play a significant role in applications such as acoustics, electromagnetics, and quantum mechanics. For scattering by a sound-soft obstacle, they are the poles of the meromorphic continuation of the scattering operator and have been a classic subject in the scattering theory \cite{LaxPhillips1989, Taylor1996, Labreuche1998siamJAM, Ralston1972CPAM, SemyonMaciej}. In contrast to the extensive theoretical work by many researchers, computational methods for scattering poles have not yet garnered much attention from the numerical analysis community. This is partly because the problem is nonlinear and the partial differential equation is defined on the unbounded domain. Since scattering poles are complex with negative imaginary parts, and their associated eigenfunctions grow exponentially, the standard radiation condition for the scattered field at positive wavenumbers is no longer applicable. The correct condition at infinity is the outgoing condition, which coincides with the radiation condition when the wavenumber is positive. As a result, care must be taken when imposing artificial boundary conditions if the unbounded domain is truncated. Artificial boundary conditions that work well for the scattering problem when the wavenumbers are positive may perform poorly in this context. In contrast, boundary integral formulations satisfy the outgoing condition naturally.
	
	There exist two main groups of numerical methods for computing scattering poles. The first group comprises finite element methods, which typically require truncation of the unbounded domain. To address this, Dirichlet-to-Neumann (DtN) mapping and the Perfectly Matched Layer (PML) technique have been employed \cite{Lenoir1992, KimPasciak2009MC, NannenWess2018BIT, XiLinSun2024, XiGongSun2024}. While PML results in a linear eigenvalue problem, it introduces spurious eigenvalues that are difficult to identify. In contrast, DtN mapping does not contaminate the spectrum and is particularly advantageous with the development of highly efficient and robust contour integral-based methods for nonlinear matrix eigenvalue problems \cite{Beyn2012LAA, HuangEtal2016JCP, SunZhou2016, ChenSunXia2024, XiSun2024arXiv}. It is worth noting that Hardy space infinite elements have also been proposed \cite{HohageNannen2009}.
	
	The second group of methods relies on boundary integral operators, which inherently satisfy the outgoing wave condition. These methods require discretization only on the boundary of the obstacle, resulting in significantly smaller algebraic systems compared to finite element methods. A Galerkin boundary element method for Dirichlet Laplacian eigenvalues is analyzed in \cite{SteinbachUnger2012}. Additionally, a combined integral equation approach for the scattering poles of a sound-hard obstacle is proposed in \cite{OlafUnger2017MMAS}, while Nystr\"{o}m's method has been applied in \cite{MaSun2023AML} to achieve highly accurate results.

	In this paper, we consider the computation of scattering poles of a sound-soft obstacle. These poles are shown to be the eigenvalues of either the single-layer operator or the sum of the identity operator and the double-layer operator \cite{Taylor1996}.
		For the single-layer operator case,  we propose a Fourier-Galerkin method to approximate that operator and prove the regular convergence of the discrete operator \cite{GongSun2023ETNA}. An error estimate is then derived using the abstract approximation theory for eigenvalue problems of holomorphic Fredholm operator functions \cite{karma1996a, karma1996b}. 
		For the case of sum of the identity operator and the double-layer operator, we introduce an auxiliary operator that shares the same eigenvalues and is demonstrated to be Fredholm with index zero. The auxiliary operator is treated in a similar manner. 
	    Numerical examples are provided to validate the effectiveness of the proposed method. 
	
	The remainder of the paper is organized as follows. In Section 2, we define the scattering pole for a sound-soft obstacle and present two equivalent eigenvalue problems involving boundary integral operators. In Section 3, we propose a Fourier-Galerkin method for these eigenvalue problems. Specifically, we demonstrate that the boundary integral operators are holomorphic and Fredholm with index zero. We establish the convergence of eigenvalues by proving the regular convergence of the Fourier-Galerkin approximations and applying Karma's abstract approximation theory. Section 4 provides implementation details and several numerical examples for validation. Finally, Section 5 presents conclusions and discusses directions for future work.

	
	\section{Integral Equations for Scattering Poles}

	We introduce in this section  the boundary integral formulations for the scattering poles of a sound soft obstacle.  	Let $\Omega\subset\mathbb{R}^2$ be a bounded simply connected  domain with boundary $\Gamma:=\partial\Omega$. The scattering problem for a sound soft obstacle $\Omega$ is to find $u\in H^{(1)}_{loc}(\mathbb{R}^2\setminus\overline{\Omega})$ such that
	\begin{subequations}\label{pro:SP}
		\begin{align}
			\Delta u +\kappa^2 u=0 \quad  ~&\text{in}~\mathbb{R}^2\setminus\overline{\Omega}, \label{eq:TEP1}\\
			u=f \quad ~&\text{on}~\Gamma,\label{eq:TEP2}\\
			\lim_{r\to\infty}\sqrt{r}\left(\frac{\partial u}{\partial r}-{\rm i}\kappa u \right)&=0, \label{eq:TEP3}
		\end{align}
	\end{subequations}
	where $\kappa\in\mathbb{C}$ is the wave number, $f\in L^2(\Gamma)$, and $r=\abs{x}$. Equation \eqref{eq:TEP3} is Sommerfeld radiation condition. Let $\mathcal{B}(\kappa): L^2(\Gamma)\mapsto H^{(1)}_{loc}(\mathbb{R}^2\setminus\overline{\Omega})$ denote the solution operator of \eqref{pro:SP} such that  $u:=\mathcal{B}(\kappa)f$ satisfies \eqref{eq:TEP1}--\eqref{eq:TEP3}. 
	
	The operator $\mathcal{B}(\kappa)$ is well-defined and holomorphic on the upper half-plane of $\mathbb{C}$ and can be meromorphically continued to the whole complex plane \cite{Taylor1996}. The poles of $\mathcal{B}(\kappa)$ in $\set{\kappa:\rm{Im}(\kappa)<0}$ are called scattering resonances or poles depending on the context \cite{Labreuche1998siamJAM,Ralston1972CPAM}. At a scattering pole, there exists a non-zero scattered field $u$ in the absence of the incident field ($f=0$).  In contrast, if there are wave numbers for which there exists an incident field that does not scatter by the scattering object, these wave numbers are some interior eigenvalues associated with the support of the scatterer \cite{CCH2020}.


	We first reformulate the scattering problem \eqref{pro:SP} using boundary integrals. Let $\Phi$ be the Green’s function given by
	\begin{equation*}
		\Phi(x,y;\kappa):=\dfrac{{\rm i}}{4}{ H}^{(1)}_0(\kappa\abs{x-y}) \quad ~\text{for}~x,y\in\mathbb{R}^2,
	\end{equation*}
	where ${H}_0^{(1)}$ is the Hankel function of the first kind of order zero. The single and double layer potentials are defined as
	\begin{equation*}\label{Sec2:singleLayer_P}
		(\mathcal{SL}(\kappa)[\phi])(x):=2\int_{\Gamma}\Phi(x,y;\kappa)\phi(y){\rm d}s(y), \quad x\in\mathbb{R}^2\backslash\Gamma,
	\end{equation*}
	and 
	\begin{equation*}\label{Sec2:doubleLayer_P}
		(\mathcal{DL}(\kappa)[\phi])(x):=2\int_{\Gamma}\dfrac{\partial\Phi(x,y;\kappa)}{\partial \vec{\nu}(y)}\phi(y){\rm d}s(y), \quad x\in\mathbb{R}^2\backslash\Gamma, 
	\end{equation*}
	where $\nu$ is the unit outward normal to $\Gamma$ and $\phi$ is an integrable function. The interior and exterior Dirichlet traces on $\Gamma$ of the single and double layer potentials satisfy the jump conditions
	\begin{equation}\label{Sec2:signleLayer_Jump}
		(\mathcal{SL}(\kappa)[\phi])^{\pm}
		=\mathcal{S}(\kappa)[\phi]
	\end{equation}
	and
	\begin{equation}\label{Sec2:doubleLayer_DJump}
		(\mathcal{DL}(\kappa)[\phi])^{\pm}
		=\mathcal{D}(\kappa)[\phi]\pm\phi,
	\end{equation}
	where the single and double layer operators are defined by
	\begin{equation*}
		(\mathcal{S}(\kappa)[\phi])(x):=2\int_{\Gamma}\Phi(x,y;\kappa)\phi(y){\rm d}s(y), \quad x\in\Gamma,
	\end{equation*}
	and 
	\begin{equation*}
		(\mathcal{D}(\kappa)[\phi])(x):=2\int_{\Gamma}\dfrac{\partial\Phi(x,y;\kappa)}{\partial \vec{\nu}(y)}\phi(y){\rm d}s(y), \quad x\in\Gamma.
	\end{equation*}

	We now present two integral formulations for the scattering problem, which respectively use the single and double layer potentials. The first one seeks the solution of \eqref{pro:SP} in the form $u=\mathcal{SL}(\kappa)\phi$, where $\phi$ is the unknown density. According to \eqref{Sec2:signleLayer_Jump}, we obtain the integral equation
	\begin{equation}\label{Sec2:BIE-S}
		\dfrac{1}{2}\mathcal{S}(\kappa)\phi=f.
	\end{equation}
	If $\mathcal{S}(\kappa)$ is invertible, the scattering operator $\mathcal{B}(\kappa)$ has the form
	\begin{equation*}
		\mathcal{B}(\kappa):=2\mathcal{SL}(\kappa)\mathcal{S}(\kappa)^{-1}.
	\end{equation*}
	Note that if $\kappa\in\mathbb{R}$, $u=\mathcal{SL}(\kappa)\phi$, not a Dirichlet eigenvalue on $\Omega$, and $\phi$ satisfies \eqref{Sec2:BIE-S}, then $u=\mathcal{SL}(\kappa)\phi$ solves the interior Dirichlet problem of the Helmholtz equation in $\Omega$.

	The second integral formulation is to seek the solution  to \eqref{pro:SP}  in the form $u=\mathcal{DL}(\kappa)\phi$. This together with \eqref{Sec2:doubleLayer_DJump} yields the integral equation
	\begin{equation}\label{Sec2:BIE-D}
		\dfrac{1}{2}\left(\mathcal{I}+\mathcal{D}(\kappa)\right)\phi=f.
	\end{equation}
	If $\mathcal{I}+\mathcal{D}(\kappa)$ is invertible, the scattering operator $\mathcal{B}(\kappa)$ has the form
	\begin{equation*}
		\mathcal{B}(\kappa):=2\mathcal{DL}(\kappa)\left(\mathcal{I}+\mathcal{D}(\kappa)\right)^{-1}.
	\end{equation*}
	
	\begin{remark}
		The solution defined in the form of single layer potential or double layer potential satisfies the outgoing condition as the Green's function for all $\kappa \in {\mathbb C}$. The outgoing condition is equivalent to the Sommerfeld radiation condition when $\kappa$ is positive, which is no longer true for $\kappa$ with a negative imaginary part. As a consequence, \eqref{eq:TEP3} cannot be used for scattering poles.
	\end{remark}
	
	It is shown in \cite{Taylor1996} that scattering poles of $\mathcal{B}(\kappa)$ are the zeros of $\mathcal{S}(\kappa)$, or $\mathcal{I}+\mathcal{D}(\kappa)$ in  $\set{\kappa:\rm{Im}(\kappa)<0}$. Hence, in this paper, we shall compute the zeros of $\mathcal{S}(\kappa)$, or $\mathcal{I}+\mathcal{D}(\kappa)$ by proposing a Fourier Galerkin method. To this end, we assume that the boundary curve $\Gamma$ is given by a $2\pi$-periodic parametric representation of the form
	\begin{equation*}
		\Gamma:=\{z(t):=(z_1(t),z_2(t))^\top: t\in I:=[0,2\pi]\}.
	\end{equation*}  
	The parameterized operator $\mathcal{S}(\kappa)$ is
	\begin{equation}\label{Sec2:singleLayer_O}
		(\mathcal{S}(\kappa)[\phi])(s)=2\int_0^{2\pi}K^S(s,t;\kappa)\varphi(t)\,{\rm d}t,~\text{for}~s\in I,
	\end{equation}
	where $\varphi(t)=\phi(z(t))$ and
	\begin{equation*}
		K^S(s, t;\kappa)=\dfrac{\rm i}{4}H_0^{(1)}(\kappa|z(s)-z(t)|)|z'(t)|.
	\end{equation*}
	The parameterized operator $\mathcal{D}(\kappa)$ is given by
	\begin{equation}\label{Sec2:doubleLayer_O}
		(\mathcal{D}(\kappa)[\phi])(s)=2\int_0^{2\pi}K^D(s,t;\kappa)\varphi(t)\,{\rm d}t,~\text{for}~s\in I,
	\end{equation}
	where
	\begin{equation*}
		K^D(s, t;\kappa)=\dfrac{\rm i\kappa}{4}\dfrac{(z(s)-z(t))\cdot\vec{\nu}(z(t))}{|z(s)-z(t)|}H_1^{(1)}(\kappa|z(s)-z(t)|)|z'(t)|
	\end{equation*}
	with ${ H}_1^{(1)}$ being the Hankel function of the first kind of order one. To analyze the Fourier Galerkin method, we need the boundary operators for Laplacian. Let $\mathcal{S}(0)$ and $\mathcal{D}(0)$ denote the single layer operator and  double layer operator for the Laplace operator, which are defined the same as $\eqref{Sec2:singleLayer_O}$and $\eqref{Sec2:doubleLayer_O}$ by replacing $\Phi$ with 
	\begin{equation*}
		\Phi_0(x,y):=-\dfrac{1}{\pi}\ln{\abs{x-y}},~\text{for}~x,y\in\mathbb{R}^2,
	\end{equation*}
	the fundamental solution of the Laplace operator. The parametric form of $\mathcal{S}(0)$ and $\mathcal{D}(0)$ can be obtained by the same parametric representation of $\Gamma$ with $\mathcal{S}(\kappa)$ and $\mathcal{D}(\kappa)$. We present some properties of those operators and refer the readers to \cite{Anne2011,DominguezLyonTurc2016,Kirsch1989,Taylor1996}. 
	
	\begin{theorem}\label{thm:Sec2}
		Let $\kappa \in \mathbb C$ and $\Gamma$ be of $C^{p+3,1}$ with $p\geq 0$. 
		\begin{enumerate}
			\item[\rm (i)] The operator $\mathcal{D}(\kappa): H^{p}(I)\to H^{p}(I)$ is compact.
			\item[\rm(ii)] The operator $\mathcal{D}(\kappa_1)-\mathcal{D}(\kappa_2): H^{p}(I)\to H^{p+1}(I) $ is continuous and compact if $\kappa_1\neq\kappa_2$.
			\item[\rm(iii)] The operator $\mathcal{I}+\mathcal{D}(\kappa): H^{p}(I)\to H^{p}(I) $ is Fredholm of index zero. Moreover, if the imaginary part of $\kappa$ is positive, the operator $\mathcal{I}+\mathcal{D}(\kappa): H^{p}(I)\to H^{p}(I) $ is invertible.
			\item[\rm(iv)] The operator $\mathcal{S}(\kappa)-\mathcal{S}(0): H^{p}(I)\to H^{p+2}(I) $ is continuous and compact.
		\end{enumerate}
	\end{theorem}
\begin{proof}
	The properties (i) and (iv) can be found in \cite{Kirsch1989}. This together with Proposition 7.1 in \cite{Taylor1996} yields (iii). According to \cite{Kirsch1989}, We have that $\mathcal{D}(\kappa_j)-\mathcal{D}(0): H^{p}(I)\to H^{p+1}(I) $ is continuous and compact for $j=1$, $2$. The property (ii) is then obtained since $\mathcal{D}(\kappa_1)-\mathcal{D}(\kappa_2)=(\mathcal{D}(\kappa_1)-\mathcal{D}(0))-(\mathcal{D}(\kappa_2)-\mathcal{D}(0))$.
\end{proof}
	
	
	\section{Fourier Galerkin Method}
	We propose a Fourier-Galerkin method to compute the zeros of $\mathcal{S}(\kappa)$ or  $\mathcal{I}+\mathcal{D}(\kappa)$, which are the scattering poles of ${\mathcal B}(\kappa)$. 
	We prove the regular convergence of the discrete approximation operators $\mathcal{S}_n(\kappa)$ to $\mathcal{S}(\kappa)$.
   For  $\mathcal{I}+\mathcal{D}(\kappa)$, we  construct a new integral operator $\mathcal{H}(\kappa)$, which has the same zeros as $\mathcal{I}+\mathcal{D}(\kappa)$, and show that it is a holomorphic Fredholm operator function.  The convergence of the eigenvalues is obtained using the abstract approximation theory for the eigenvalue problem of a holomorphic Fredholm operator function. 
	
	We first recall Sobolev space $H^p(I)$ for $p\geq 0$ defined by requiring for their elements a certain decay of the Fourier coefficients.  Let
	\begin{equation}\label{Sec3:TriangularBasisFunctions}
		e_m(t):=\dfrac{1}{\sqrt{2\pi}}{\rm e}^{{\rm i}m t},~t\in I,
	\end{equation}
	for $m\in\mathbb{Z}:=\{\ldots,-1,0,1,\ldots\}$. For $p\geq0$, we denote by $H^p(I)$ the space of  all functions $\phi\in L^2(I)$ such that 
	\[
	\sum_{m\in\mathbb{Z}}(1+m^2)^p\abs{\phi_m}^2<\infty,
	\] 
	where $\phi_m:=\int_I\phi(t) \overline{e_m(t)}{\rm d}t, m\in\mathbb{Z},$ are the Fourier coefficients of $\phi$. The norm on $H^p(I)$ is given by $\norm{\cdot}_p:=\langle{\cdot,\cdot}\rangle_p^{1/2}$, where the inner product is defined by
	\begin{equation*}
		\langle{\phi,\psi}\rangle_p:=\sum_{m\in\mathbb{Z}}(1+m^2)^p\phi_m\overline{\psi_m},~\phi,\psi\in H^p(I).
	\end{equation*}
	Note that $H^0(I)$ coincides with $L^2(I)$.  For each $n\in\mathbb{N}$, we define a finite-dimensional subspace 
	\[
	\mathbb{T}_n:=\text{span}\{e_m:\abs{m}\in\mathbb{Z}_n\}, \quad \mathbb{Z}_n:=\set{0,1,\ldots,n}.
	\] Let $\mathcal{P}_n$ be the orthogonal projection from $L^2(I)$ to  $\mathbb{T}_n$. It is well known that there exists a positive constant $c$ such that for any $\phi\in H^r(I)$ with $r>0$ and any $m\in[0,r)$  \cite{kressBook1989}
	\begin{equation}\label{sec3:proj}
		\norm{\phi-\mathcal{P}_n\phi}_m\leq cn^{m-r}\norm{\phi}_r
	\end{equation}
	and
	\begin{equation}\label{sec3:proj0}
		\lim_{n\to\infty}\norm{\phi-\mathcal{P}_n\phi}_r=0.
	\end{equation}

	We now present some preliminaries on the abstract approximation theory for eigenvalue problems of holomorphic Fredholm operator functions \cite{karma1996a, karma1996b}. Let $\mathbb{X}$ and $\mathbb{Y}$ be complex Banach spaces. Let $\Theta\subset\mathbb{C}$ be a compact region. Assume that  $\mathcal{F}(\lambda):\mathbb{X}\mapsto\mathbb{Y}$ is a holomorphic operator function on $\Theta$ and for each $\kappa \in \Theta$, $\mathcal{F}(\kappa)$ is a Fredholm operator of index 0. The eigenvalue problem of $\mathcal{F}(\cdot)$ is to find $\lambda\in \Theta $ and $w\in\mathbb{X}$ with $w\neq0$ such that
	\begin{equation*}
		\mathcal{F}(\lambda)w=0.
	\end{equation*}
	The resolvent set  of $\mathcal{F}$ is defined as
	\begin{equation*}
		\rho(\mathcal{F}):=\{\lambda\in\Theta: \mathcal{F}(\lambda)^{-1}~ \text{exists and is bounded}\}.
	\end{equation*}
	Assume that $\rho(\mathcal{F})\neq\varnothing$. Then the spectrum $\sigma(\mathcal{F}):=\Theta\setminus \rho(\mathcal{F})$ of $\mathcal{F}$ has no cluster points in $\Theta$ and all the elements in $\sigma(\mathcal{F})$ are eigenvalues of $\mathcal{F}$. 
	
	To approximate the eigenvalues of $\mathcal{F}(\cdot)$, we need two sequences of discrete Banach spaces $\mathbb{X}_n$ and  $\mathbb{Y}_n$, and a sequences of discrete operator functions $\mathcal{F}_n(\cdot): \mathbb{X}_n \to \mathbb{Y}_n$ for $n\in\mathbb{N}$ such that the following approximation properties hold.
	\begin{enumerate}
		\item[\rm(A1)] { There exist linear bounded mappings $\mathcal{L}_n:\mathbb{X}\mapsto\mathbb{X}_n$  and $\mathcal{Q}_n:\mathbb{Y}\mapsto\mathbb{Y}_n$ such that
			\begin{equation*}
				\lim_{n\to\infty}\norm{\mathcal{L}_nw}_{\mathbb{X}_n}=\norm{w}_{\mathbb{X}},~w\in\mathbb{X},~\lim_{n\to\infty}\norm{\mathcal{Q}_nw}_{\mathbb{Y}_n}=\norm{w}_{\mathbb{Y}},~w\in\mathbb{Y}.
		\end{equation*} }
		\item[\rm(A2)] { $\{\mathcal{F}_n(\cdot)\}_{n\in\mathbb{N}}$ is equibounded on $\Theta$.}
		\item[\rm(A3)] { For each $\kappa\in\Theta$, $\{\mathcal{F}_n(\kappa)\}_{n\in\mathbb{N}}$ approximates $\mathcal{F}(\kappa)$, i.e., for $w\in\mathbb{X}$,
			\begin{equation*}
				\lim_{n\to\infty}\norm{\mathcal{F}_n(\kappa)\mathcal{L}_nw-\mathcal{Q}_n\mathcal{F}(\kappa)w}_{\mathbb{Y}_n}=0.
		\end{equation*}}
		\item[\rm(A4)] { For each $\kappa\in\Theta$, $\{\mathcal{F}_n(\kappa)\}_{n\in\mathbb{N}}$ is regular, i.e.,  if $\{\mathcal{F}_n(\kappa)x_n\}_{n\in\mathbb{N}}$ is compact for $\norm{x_n}_{\mathbb{X}_n}\leq1$ with $n\in\mathbb{N}$, then $\{x_n\}_{n\in\mathbb{N}}$ is compact.}
	\end{enumerate}
	
	Let $\lambda\in\sigma(\mathcal{F})$ and $\tau(\mathcal{F},\lambda)$ be the ascent of $\lambda$, i.e., the maximal length of all Jordan chains for $\lambda$. Denote by $\Lambda(\lambda)$ the closed linear hull of the generalized eigenfunctions associated to $\lambda$. Under the assumption of above properties, the following theorem states that all eigenvalues and eigenfunctions of $\mathcal{F}(\cdot)$ are approximated correctly by those of $\mathcal{F}_n(\cdot)$ (see, e.g., Theorem 2.10 in \cite{Beyn2014}).
	
	\begin{theorem}\label{sec3:thm0}
		For any $\lambda\in\sigma(\mathcal{F})$, there exists a positive integer $N\in\mathbb{N}$, a positive constant $C$ and a sequence $\lambda_n\in\sigma(\mathcal{F}_n)$ 
		for $n\geq N$, such that
		$\lambda_n\to\lambda$ as $n\to\infty$ and 
		\begin{equation*}
			\abs{\lambda_n-\lambda}\leq C\epsilon_n^{1/\tau},
		\end{equation*}
		\begin{equation*}
			\inf\limits_{v\in{\rm Ker}(\mathcal{F}(\lambda))}\{\norm{v_n^0-\mathcal{L}_nv}_{\mathbb{X}_n}\}\leq C \epsilon_n^{1/\tau},
		\end{equation*}
		where $v_n^0\in{\rm Ker}(\mathcal{F}_n(\lambda_n))$ with $\norm{v_n^0}_{\mathbb{X}_n}=1$, and 
		\begin{equation*}
			\epsilon_n=\max\limits_{\abs{\eta-\lambda}\leq\delta}\max\limits_{v\in\Lambda(\lambda),\norm{v}_{\mathbb{X}}=1}\{\norm{\mathcal{F}_n(\eta)\mathcal{L}_nv-\mathcal{Q}_n\mathcal{F}(\eta)v}_{\mathbb{Y}_n}\}.
		\end{equation*}	
	\end{theorem}	
	
	In the rest of this section, we shall employ the above theorem to prove that the convergence eigenvalues of the discrete operators converge. 
	
	We first consider the eigenvalue problem for $\mathcal{S}(\kappa)$.
	Write  $\mathcal{S}(\kappa)$ as
	\begin{equation*}\label{eq:1thm1}
		\mathcal{S}(\kappa):=\mathcal{K}+[\mathcal{S}(\kappa)-\mathcal{K}],
	\end{equation*}
	where the operator $\mathcal{K}$ is defined by
	\begin{equation*}\label{operator_K}
		\mathcal{K}[\varphi](s):=-\dfrac{1}{2\pi}\int_0^{2\pi}
		\left[\ln{\left(4\sin^2\dfrac{s-t}{2}\right)}-1\right]\varphi(t){\rm d}t.
	\end{equation*}
	The discrete operators $\mathcal{S}_n(\kappa)$ that approximate $\mathcal{S}(\kappa)$ are defined by
	\begin{equation}\label{SnKappa}
		\mathcal{S}_n(\kappa):=\mathcal{P}_n\mathcal{S}(\kappa)\mathcal{P}_n.
	\end{equation} 
	Let $\mathcal{K}_n:=\mathcal{P}_n\mathcal{K}\mathcal{P}_n$ and $\mathcal{G}_n(\kappa):=\mathcal{P}_n\mathcal{G}(\kappa)\mathcal{P}_n$ with $\mathcal{G}(\kappa):=\mathcal{S}(\kappa)-\mathcal{K}$. We present some properties of the operators $\mathcal{K}$, $\mathcal{K}_n$,  $\mathcal{G}(\kappa)$ and $\mathcal{G}_n(\kappa)$ in the following two lemmas.
	
	\begin{lemma}\label{sec3:lem1}
		Let  $p\geq0$.
				\begin{enumerate}
			\item[\rm (i)]For any $\phi\in H^{p}(I)$, $\norm{\phi}_{p}\leq\norm{\mathcal{K}\phi}_{p+1}\leq 2\norm{\phi}_{p}$. 
			\item[\rm(ii)] $\norm{\mathcal{P}_n\phi}_{p}\leq\norm{\mathcal{K}_n\phi}_{p+1}\leq 2\norm{\mathcal{P}_n\phi}_{p}$ for each $\phi\in H^{p}(I)$.
			\item[\rm(iii)] For any $\phi\in H^{p}(I)$,
			$$\lim_{n\to\infty}\norm{(\mathcal{K}_n-\mathcal{K})\phi}_{p+1}=0.$$
			\item[\rm(iv)] If $\Gamma$ is of $C^{p+3,1}$,
			the operator $\mathcal{S}(0)-\mathcal{K}$ is bounded from $H^{p}(I)$ to $H^{p+2}(I)$.
		\end{enumerate}
	\end{lemma}	
	\begin{proof}
    The first conclusion {\rm (i)} can be found in the proof of  Theorem 3.18 in \cite{Kirsch1996}. 	Note that	
	the eigenvalues of $\mathcal{K}$ are $\set{1,\frac{1}{\abs{n}};n=\pm1,\pm2,\cdots}$ with
	\begin{equation*}
		\mathcal{K}e_0(t)=e_0(t),~\mathcal{K}e_n(t)=\frac{1}{\abs{n}}e_n(t),~\text{for}~n\in\mathbb{Z}\setminus\{0\},
	\end{equation*}
	which were shown in (3.65a) and (3.65b) in \cite{Kirsch1996}. Let $\phi\in H^{p}(I)$ and write
	$\phi=\sum_{n\in\mathbb{Z}}a_ne_n$.
    We have that 
    \begin{equation*}
	\mathcal{K}_n\phi= a_0e_0+\sum_{n\in\mathbb{Z}_n\setminus\{0\}}\dfrac{a_n}{\abs{n}}e_n.
    \end{equation*}
    This together with 
    \begin{equation*}
    (1+n^2)^p\leq\dfrac{(1+n^2)^{p+1}}{n^2}\leq 2(1+n^2)^p
    \end{equation*}
   yields {\rm (ii)}.
   Applying $\norm{(\mathcal{K}_n-\mathcal{K})\phi}_{p+1}\leq 2\norm{(\mathcal{I}-\mathcal{P}_n)\phi}_{p}$ with \eqref{sec3:proj0}, we obtain {\rm (iii)}. The proof of {\rm (iv)} is given by Theorem A.45 in \cite{Kirsch1996}.
   \end{proof}

	\begin{lemma}\label{sec3:lem2}
		Let $\Gamma$ be of $C^{p+3,1}$ for $p\geq1$.  
    $\mathcal{G}(\kappa):H^p(I)\mapsto H^{p+2}(I)$ is bounded, and $\mathcal{G}(\kappa):H^p(I)\mapsto H^{p+1}(I)$ is compact. Moreover,
	there exist an integer $n_0\in\mathbb{N}$ and a positive constant $c$ such that, for all $n\in\mathbb{N}$ with $n>n_0$ and $\phi\in H^{p}(I)$,
	\begin{equation}\label{convergenceOrderGnE} 
		\norm{[\mathcal{G}_n(\kappa)-\mathcal{G}(\kappa)]\phi}_{2}\leq c n^{-p}\norm{\phi}_{p}.
	\end{equation}
	\end{lemma}
	
	\begin{proof}
		 Due to (iv) of Theorem \ref{thm:Sec2} and Lemma \ref{sec3:lem1}, $\mathcal{S}(\kappa)-\mathcal{S}(0)$ and
		 $\mathcal{S}(0)-\mathcal{K}$ are bounded from $H^{p}(I)$ to $H^{p+2}(I)$.
		  This leads to the boundedness of $\mathcal{G}(\kappa)$  from $H^{p}(I)$ to $H^{p+2}(I)$ and 
		 the compactness of $\mathcal{G}(\kappa)$  from $H^{p}(I)$ to $H^{p+1}(I)$. 
		
		For $\phi\in H^{1}(I)$,
		\begin{equation*}
			[\mathcal{G}(\kappa)-\mathcal{G}_n(\kappa)]\phi=[\mathcal{G}(\kappa)-\mathcal{P}_n\mathcal{G}(\kappa)]\phi+[\mathcal{P}_n\mathcal{G}(\kappa)-\mathcal{G}_n(\kappa)]\phi.
		\end{equation*}
	    Due to Theorem \ref{thm:Sec2} and \eqref{sec3:proj}, 
	    there exist an integer $n_0\in\mathbb{N}$ and a positive constant $c$ such that, for all $n\in\mathbb{N}$ with $n>n_0$ and $\phi\in H^{p}(I)$,
	    \begin{align*}
	    	\norm{[\mathcal{G}(\kappa)-\mathcal{P}_n\mathcal{G}(\kappa)]\phi}_{2}
	    	=& \, \norm{[\mathcal{I}-\mathcal{P}_n]\mathcal{G}(\kappa)\phi}_{2} \\
	    	\leq &\, c n^{-p}\norm{\mathcal{G}(\kappa)\phi}_{p+2}\\
	    	\leq&\, c n^{-p}\norm{\phi}_{p}
	    \end{align*} 
	    and 
	    \begin{align*}
	    	\norm{[\mathcal{P}_n\mathcal{G}(\kappa)-\mathcal{G}_n(\kappa)]\phi}_{2} = &\norm{[\mathcal{P}_n\mathcal{G}(\kappa)-\mathcal{P}_n\mathcal{G}(\kappa)\mathcal{P}_n]\phi}_{2} \\
	    	\leq &\, c\norm{\mathcal{G}(\kappa)[\mathcal{I}-\mathcal{P}_n]\phi}_{2}\\
	    	\leq&\, c \norm{[\mathcal{I}-\mathcal{P}_n]\phi}_{0} \\
	    	\leq &\,c n^{-p}\norm{\phi}_{p}.
	    \end{align*} 
	    This yields \eqref{convergenceOrderGnE} and the proof is complete.

	\end{proof}
	
	 We now show the convergence of the Fourier Galerkin method for $\mathcal{S}(\kappa)$, i.e., the convergence of eigenvalues of $\mathcal{S}_n(\kappa)$ to those of $\mathcal{S}(\kappa)$.

	\begin{theorem}\label{sec3:thm1}
		Let $\Gamma$ be of $C^{p+3,1}$ for $p\geq1$. $\mathcal{S}(\kappa):H^1(I)\mapsto H^2(I)$ is Fredholm of index zero. For $\lambda\in\sigma(\mathcal{S}(\kappa))$,
		there exist an integer $n_0\in\mathbb{N}$, a positive constant $c$  and a sequence $\lambda_n\in\sigma(\mathcal{S}_n(\kappa))$ for all $n\in\mathbb{N}$ with $n>n_0$, such that, 
		$\lambda_n\to\lambda$ as $n\to\infty$ and 
		\begin{equation*}
			\abs{\lambda_n-\lambda}\leq cn^{-p/\tau},
		\end{equation*}
		\begin{equation*}
			\inf\limits_{v\in{\rm Ker}(\mathcal{S}(\lambda))}\{\norm{v_n^0-\mathcal{P}_nv}_{1}\}\leq c n^{-p/\tau},
		\end{equation*}
		where $v_n^0\in{\rm Ker}(\mathcal{S}_n(\lambda_n))$ with $\norm{v_n^0}_{1}=1$, and $\tau$ is the ascent of $\lambda$.
	\end{theorem}
	\begin{proof}
		
		Let $\Theta$ be compact and $\kappa \in \Theta$. Since $\mathcal{G}(\cdot)$ is compact from $H^{1}(I)$ to $H^{2}(I)$, and  $\mathcal{K}$ is a bounded invertible operator from $H^{1}(I)$ to $H^{2}(I)$, $\mathcal{S}(\cdot)$ is Fredholm of index zero on $\Theta$.
		
		It is clear that (A1)-(A3) hold due to the property of $\mathcal{P}_n$, the definition of $\mathcal{S}_n$, and Lemma \ref{sec3:lem2}. By Theorem 2.9 in \cite{GongSun2023ETNA}, Lemma \ref{sec3:lem1} and Lemma \ref{sec3:lem2}, we have that (A4) holds, i.e., the regular convergence of $\mathcal{S}_n$ to $\mathcal{S}$.
		Note that 
		\begin{align*}
			\epsilon_n&=\max\limits_{\abs{\eta-\lambda}\leq\delta}\max\limits_{v\in\Lambda(\lambda),\norm{v}_{1}=1}\{\norm{\mathcal{S}_n(\eta)\mathcal{P}_nv-\mathcal{P}_n\mathcal{S}(\eta)v}_{2}\}\\
			&=\max\limits_{\abs{\eta-\lambda}\leq\delta}\max\limits_{v\in\Lambda(\lambda),\norm{v}_{1}=1}\{\norm{\mathcal{G}_n(\eta)\mathcal{P}_nv-\mathcal{P}_n\mathcal{G}(\eta)v}_{2}\} \\
			& \le c n^{-p}
		\end{align*}
		due to the fact that $\Lambda(\lambda)$ is finite-dimensional and thus the two norms $\|\cdot\|_1$ and $\|\cdot \|_p$ are equivalent.
		 Then the estimates follow from Theorem \ref{sec3:thm0}.
	\end{proof}

	\begin{remark}
		In \cite{SteinbachUnger2012}, a boundary element method is proposed to compute the Dirichlet Laplacian eigenvalues, which are the real eigenvalues of $\mathcal{S}(\cdot)$. It is also mentioned that the same method can be used to compute the scattering poles, which are complex eigenvalues of $\mathcal{S}(\cdot)$. 
	\end{remark}
	
	We move on to the eigenvalue problem for $\mathcal{I}+\mathcal{D}(\kappa)$. Assume that  $\Gamma$ is of $C^{p+3,1}$  with $p\geq1$. Let $\kappa_+$ be a complex number with positive imaginary part. From  Theorem \ref{thm:Sec2}, $(\mathcal{I}+\mathcal{D}(\kappa_+))^{-1}:H^p(I)\mapsto H^p(I)$ is bounded. Define
	\begin{equation*}
		\mathcal{H}(\kappa):=\mathcal{I}+\mathcal{T}(\kappa),
	\end{equation*}
	where
	\begin{equation*}
		\mathcal{T}(\kappa):=(\mathcal{I}+\mathcal{D}(\kappa_+))^{-1}[\mathcal{D}(\kappa)-\mathcal{D}(\kappa_+)].
	\end{equation*}
	Note that  $\mathcal{H}(\kappa)$ and $\mathcal{I}+\mathcal{D}(\kappa)$ have the same eigenvalues due to the fact that
	\[
	\mathcal{H}(\kappa) = (\mathcal{I}+\mathcal{D}(\kappa_+))^{-1}(\mathcal{I}+\mathcal{D}(\kappa)).
	\]
	Consequently, we can approximate the eigenvalues of $\mathcal{H}(\kappa)$ by the Fourier Galerkin method, i.e., projecting the operator $\mathcal{T}(\kappa)$  onto the space $\mathbb{T}_n$. Let $\mathcal{T}_n(\kappa):=\mathcal{P}_n\mathcal{T}(\kappa)\mathcal{P}_n$. 	
	\begin{lemma}\label{sec3:lem3}
		Let $\Gamma$ be of $C^{p+3,1}$  with $p\geq1$. 
		\begin{itemize}
			\item[(1)] $\mathcal{T}(\kappa):H^p(I)\mapsto H^p(I)$ is compact. 
			\item[(2)] There exist an integer $n_0\in\mathbb{N}$ and a positive constant $c$ such that, for all $n\in\mathbb{N}$ with $n>n_0$ and $\phi\in H^{p}(I)$,
           \begin{equation*}\label{convergenceOrderRnE} 
	        \norm{[\mathcal{T}_n(\kappa)-\mathcal{T}(\kappa)]\phi}_{1}\leq c n^{-p}\norm{\phi}_{p}.
          \end{equation*}	
		\end{itemize}
		
	\end{lemma}
	
	\begin{proof}
		(1) According Theorem \ref{thm:Sec2}, the compactness of $\mathcal{T}(\kappa)$ follows from the compactness of  the  operator $\mathcal{D}(\kappa)-\mathcal{D}(\kappa_+)$ from $H^{p}(I)$ to $H^{p}(I)$ and boundedness of the operator $(\mathcal{I}+\mathcal{D}(\kappa_+))^{-1}:H^p(I)\mapsto H^p(I)$ for $\kappa_+$ with positive imaginary part.
		
		(2) Note that, for $\phi\in H^{1}(I)$,
		\begin{equation*}
			[\mathcal{T}(\kappa)-\mathcal{T}_n(\kappa)]\phi=[\mathcal{T}(\kappa)-\mathcal{P}_n\mathcal{T}(\kappa)]\phi+[\mathcal{P}_n\mathcal{T}(\kappa)-\mathcal{T}_n(\kappa)]\phi.
		\end{equation*}
		Due to Theorem \ref{thm:Sec2} and \eqref{sec3:proj},
		there exist an integer $n_0\in\mathbb{N}$ and a positive constant $c$ such that, for all $n\in\mathbb{N}$ with $n>n_0$ and $\phi\in H^{p}(I)$,
	    \begin{align*}
	    	\norm{[\mathcal{T}(\kappa)-\mathcal{P}_n\mathcal{T}(\kappa)]\phi}_{1}
	    	=&\, \norm{[\mathcal{I}-\mathcal{P}_n]\mathcal{T}(\kappa)\phi}_{1} \\
	    	\leq&\, c n^{-p}\norm{\mathcal{T}(\kappa)\phi}_{p+1}\\
	    	\leq&\, c n^{-p}\norm{[\mathcal{D}(\kappa)-\mathcal{D}(\kappa_+)]\phi}_{p+1} \\
	    	\leq&\, c n^{-p}\norm{\phi}_{p}
	    \end{align*} 
	    and
	    \begin{align*}
	    	\norm{[\mathcal{P}_n\mathcal{T}(\kappa)-\mathcal{T}_n(\kappa)]\phi}_{1}
	    	\leq&\, c\norm{\mathcal{T}(\kappa)[\mathcal{I}-\mathcal{P}_n]\phi}_{1}\\
	    	\leq&\, c \norm{[\mathcal{D}(\kappa)-\mathcal{D}(\kappa_+)][\mathcal{I}-\mathcal{P}_n]\phi}_{1}\\
	    	\leq&\, c \norm{[\mathcal{I}-\mathcal{P}_n]\phi}_{0} \\
	    	\leq& \,c n^{-p}\norm{\phi}_{p}.
	    \end{align*} 
	    The proof is complete.

	\end{proof}

	We now present the convergence theorem for approximating the eigenvalues of $\mathcal{H}(\kappa)$. The discrete operator is defined as
	\begin{equation*}
		\mathcal{H}_{n}(\kappa):=\mathcal{I}_n+\mathcal{T}_{n}(\kappa),
	\end{equation*}
	where $\mathcal{I}_n:=\mathcal{P}_n\mathcal{I}\mathcal{P}_n$.
	
	\begin{theorem}\label{sec3:thm2}
		Let $\Gamma$ be of $C^{p+3,1}$ for $p\geq1$. $\mathcal{H}(\kappa):H^1(I)\mapsto H^1(I)$ is Fredholm of index zero. For $\lambda\in\sigma(\mathcal{H})$,
		there exist an integer $n_0\in\mathbb{N}$, a positive constant $c$  and a sequence $\lambda_n\in\sigma(\mathcal{H}_n)$ for all $n\in\mathbb{N}$ with $n>n_0$, such that, 
		$\lambda_n\to\lambda$ as $n\to\infty$ and 
		\begin{equation*}
			\abs{\lambda_n-\lambda}\leq cn^{-p/\tau},
		\end{equation*}
		\begin{equation*}
			\inf\limits_{v\in{\rm Ker}(\mathcal{H}(\lambda))}\{\norm{v_n^0-\mathcal{P}_nv}_{1}\}\leq c n^{-p/\tau},
		\end{equation*}
		where $v_n^0\in{\rm Ker}(\mathcal{H}_n(\lambda_n))$ with $\norm{v_n^0}_{1}=1$, and $\tau$ is the ascent of $\lambda$.
	\end{theorem}
	\begin{proof}
		
		Let $\Theta$ be compact and $\kappa \in \Theta$. Since  the  operator $\mathcal{T}(\kappa)$ is compact from $H^{1}(I)$ to $H^{1}(I)$, $\mathcal{H}(\kappa)$ is Fredholm of index zero.
		
		It is clear that (A1)-(A3) hold due to the property of $\mathcal{P}_n$, the definition of $\mathcal{T}_n$, and Lemma \ref{sec3:lem3}. By Theorem 2.9 in \cite{GongSun2023ETNA}, we have that (A4) holds, i.e., $\mathcal{H}_n$ converges to $\mathcal{H}$ regularly. Then the error estimates follow from Theorem~\ref{sec3:thm0}.
	\end{proof}
	
	We have proved the convergence of the eigenvalues of $\mathcal{H}_{n}(\kappa)$ to those of $\mathcal{H}(\kappa)$, which are the eigenvalues of $\mathcal{I}+\mathcal{D}(\kappa)$. In the rest of this section, we show that, for $n$ large enough, the eigenvalues of $\mathcal{H}_{n}(\kappa)$ are the same as $[\mathcal{P}_n(\mathcal{I}+\mathcal{D}(\kappa))\mathcal{P}_n]$. Thus in practice, we compute the eigenvalues of $[\mathcal{P}_n(\mathcal{I}+\mathcal{D}(\kappa))\mathcal{P}_n]$, which is simpler to implement.

	Let $\mathcal{F}(\kappa_+):=\mathcal{I}+\mathcal{D}(\kappa_+)$, $\mathcal{F}_n(\kappa_+):=\mathcal{P}_n[\mathcal{I}+\mathcal{D}(\kappa_+)]\mathcal{P}_n$ and $\mathcal{D}_n(\kappa_+):=\mathcal{P}_n\mathcal{D}(\kappa_+)\mathcal{P}_n$.
	\begin{lemma}\label{sec3:lem4}
	Let $\Gamma$ be of $C^{p+3,1}$ with $p\geq 1$. There exists an integer $n_0\in\mathbb{N}$ such that, for all $n\in\mathbb{N}$ with $n>n_0$, $\mathcal{I}+\mathcal{D}_n(\kappa_+)$ is invertible.
	\end{lemma}
	\begin{proof}
		According to (iii) of Theorem \ref{thm:Sec2}, we have that $\mathcal{F}^{-1}(\kappa_+)$ is bounded, and 
		\begin{equation*}
		\mathcal{F}^{-1}(\kappa_+)=\mathcal{I}-\mathcal{F}^{-1}(\kappa_+)\mathcal{D}(\kappa_+). 
		\end{equation*}
	 Let $\mathcal{J}_n(\kappa_+):=\mathcal{I}-\mathcal{F}^{-1}(\kappa_+)\mathcal{D}_n(\kappa_+)$.
    Simple calculation yields that 
	 \begin{equation}
	 	\mathcal{J}_n(\kappa_+)(\mathcal{I}+\mathcal{D}_n(\kappa_+))=\mathcal{I}-\mathcal{W}_n,
	 \end{equation}
 where
 \begin{equation*}
 	\mathcal{W}_n(\kappa)=\mathcal{F}^{-1}(\kappa_+)[\mathcal{D}_n(\kappa_+)-\mathcal{D}(\kappa_+)]\mathcal{D}_n(\kappa_+).
 \end{equation*}
Since $\norm{[\mathcal{D}_n(\kappa_+)-\mathcal{D}(\kappa_+)]\mathcal{D}_n(\kappa_+)}\to 0$ as $n\to \infty$. There exists an integer $n_0\in\mathbb{N}$ such that, for all $n\in\mathbb{N}$ with $n>n_0$, $\norm{\mathcal{W}_n}<1$. This yields that
$\mathcal{I}+\mathcal{D}_n(\kappa_+)$ is invertible for $n > n_0$.
	\end{proof}
	
	If $\mathcal{I}+\mathcal{D}_n(\kappa_+)$ is invertible, $\mathcal{F}_n(\kappa_+)$ is invertible. Note that
	\begin{equation}\label{eigsRelationship}
		\mathcal{I}_n+\mathcal{F}_n^{-1}(\kappa_+)\mathcal{P}_n[\mathcal{D}(\kappa)-\mathcal{D}(\kappa_+)]\mathcal{P}_n=\mathcal{F}_n^{-1}(\kappa)[\mathcal{P}_n(\mathcal{I}+\mathcal{D}(\kappa))\mathcal{P}_n].
	\end{equation}
	The left hand side of the above equation is $\mathcal{H}_{n}(\kappa)$. Thus for $n$ large enough, $\mathcal{H}_{n}(\kappa)$ and $[\mathcal{P}_n(\mathcal{I}+\mathcal{D}(\kappa))\mathcal{P}_n]$ have the same eigenvalues.

	\section{Numerical Implementation and Examples}
	In this section we present the detail of the Fourier Galerkin method and several examples. 
	We first discretize the integral operators $\mathcal{S}(\kappa)$ and $\mathcal{D}(\kappa)$ defined by \eqref{Sec2:singleLayer_O} and \eqref{Sec2:doubleLayer_O}, respectively. The trigonometric projection is as follows. For given $n_1$,$n_2\in \mathbb{N}$, we denote by $\mathbb{T}_{n_1}\bigotimes\mathbb{T}_{n_2}$ the tensor product space of $2\pi$-biperiodic trigonometric functions. Choose an equidistant set of knots 
	\[
	s_j :=2\pi j/(2n_1+1), \quad t_k:=2\pi k/(2n_2+1),\quad n_1, n_2\in\mathbb{N}, j\in\mathbb{Z}_{2n_1}, k\in\mathbb{Z}_{2n_2}.
	\] 
	The trigonometric interpolation projection
	$\mathcal{Q}_{n_1,n_2}: [C(I)]^2\mapsto\mathbb{T}_{n_1}\bigotimes\mathbb{T}_{n_2}$ is defined as
	\begin{equation*}
		(\mathcal{Q}_{n_1,n_2}[f])(s_j,t_k)=f(s_j,t_k),~\text{for}~j\in\mathbb{Z}_{2n_1}, k\in\mathbb{Z}_{2n_2}.
	\end{equation*}
	It holds that 
	\begin{equation*}\label{sec4:InterpolationPro}
		(\mathcal{Q}_{n_1,n_2}[f])(s,t)=\sum_{\abs{j}\in\mathbb{Z}_{n_1}}\sum_{\abs{k}\in\mathbb{Z}_{n_2}}c_{j,k}e_j(s)e_k(t),
	\end{equation*} 
	where $e_j$ is defined by \eqref{Sec3:TriangularBasisFunctions}  and
	\begin{equation*}
		c_{j,k}=\dfrac{4\pi^2}{(2n_1+1)(2n_2+1)}\sum_{\ell\in\mathbb{Z}_{2n_1}}
		\sum_{m\in\mathbb{Z}_{2n_2}}f(s_\ell,t_m)\overline{e_{j}(s_\ell)}\overline{e_{k}(t_m)}.
	\end{equation*}

	We then split the operators  $\mathcal{S}(\kappa)$ and $\mathcal{D}(\kappa)$ into two types of integral operators. The first type is defined as
	\begin{equation*}\label{Aphi}
		(\mathcal{A}(a)[\phi])(s)=\int_{I}a(s,t)\ln{\left(4\sin^2{\dfrac{t-s}{2}}\right)}\phi(t){\rm d}t,~s\in I,
	\end{equation*}
	with a smooth kernel $a$. The second type is given by
	\begin{equation*}\label{Bphi}
		(\mathcal{B}(b)[\phi])(s)=\int_{I}b(s,t)\phi(t){\rm d}t,~s\in I,
	\end{equation*}
	with a smooth kernel $b$. Let $J_n$ denote the Bessel function of the first kind of order $n$ and $C_E$ denote Euler's constant. The parameterized operator $\mathcal{S}(\kappa)$ is given by
	\begin{equation*}
		(\mathcal{S}(\kappa)[\phi])(s)=(\mathcal{A}(a_S)[\phi])(s)+(\mathcal{B}(b_S)[\phi])(s), \quad  s\in I,
	\end{equation*}
	where
	\begin{align*}
		a_S(s, t)= & -\dfrac{1}{2\pi}J_0(\kappa|z(s)-z(t)|)|z'(t)|,\\
		b_S(s, t)= &\dfrac{\rm i}{2}H_0^{(1)}(\kappa|z(s)-z(t)|)|z'(t)|-a_S(s, t)\ln\left(4\sin^2\dfrac{s-t}{2}\right).
	\end{align*}
	The diagonal terms are 
	\begin{align*}
		a_S(t, t)= &\, -\dfrac{1}{2\pi}|z'(t)|, \\
		b_S(t, t)= &\,\left(\dfrac{\rm i}{2}-\dfrac{C_E}{\pi}-\dfrac{1}{2\pi}\ln{\left(\dfrac{\kappa^2}{4}|z'(t)|^2\right)} \right)|z'(t)|.
	\end{align*}
	The parameterized operator $\mathcal{D}(\kappa)$ has the form
	\begin{equation*}
		(\mathcal{D}(\kappa)[\phi])(s)=(\mathcal{A}(a_D)[\phi])(s)+(\mathcal{B}(b_D)[\phi])(s), \quad s\in I,
	\end{equation*}
	where 
	\begin{align*}
		a_D(s, t) =&- \dfrac{\kappa}{2\pi}\dfrac{(z(s)-z(t))\cdot\nu(z(t))}{|z(s)-z(t)|}J_1(\kappa|z(s)-z(t)|)|z'(t)|,\\
		b_D(s, t)  =&\dfrac{\mathrm{i}\kappa}{2}\dfrac{(z(s)-z(t))\cdot\nu(z(t))}{|z(s)-z(t)|}H_1^{(1)}(\kappa|z(s)-z(t)|)|z'(t)|\\
		&-a_D(s, t)\ln\left(4\sin^2\dfrac{s-t}{2}\right).
	\end{align*}
	The diagonal term is
	\begin{equation*}
		a_D(t, t)=0,~
		b_D(t, t)=\dfrac{1}{2\pi}\dfrac{z''(t)\cdot\nu(z(t))}{|z'(t)|}.
	\end{equation*}
	
	We finally  use the matrices
	\begin{equation*}
		\vec{S}_n(\kappa):=\mathcal{P}_n\mathcal{A}(\mathcal{Q}_{n,n}a_S)\mathcal{P}_n+\mathcal{P}_n\mathcal{B}(\mathcal{Q}_{n,n}b_S)\mathcal{P}_n
	\end{equation*}
	and
	\begin{equation*}
		\vec{D}_n(\kappa):=\mathcal{P}_n\mathcal{A}(\mathcal{Q}_{n,n}a_D)\mathcal{P}_n+\mathcal{P}_n\mathcal{B}(\mathcal{Q}_{n,n}b_D)\mathcal{P}_n
	\end{equation*}
	for $n\in\mathbb{N}$ to approximate the operators $\mathcal{S}(\kappa)$ and $\mathcal{D}(\kappa)$, respectively. The computation for the elements of those matrices can be found in \cite{MaSun2023JSC}.
	
	
	To compute the eigenvalues of the nonlinear matrices $\vec{S}_n(\kappa)$ or $\vec{I}_{2n+1}+\vec{D}_n(\kappa)$ for $n\in\mathbb{N}$,  we employ the parallel spectral indicator method (SIM) \cite{XiSun2024arXiv} (see also \cite{Beyn2012LAA, Beyn2014, HuangEtal2016JCP, Huang2018}). SIM uses contour integrals and is easy to implement. Its idea can be explained as follows. Let $\gamma:=\{z=z_0+r{\rm e}^{{\rm i }\theta}:~\theta\in[0,2\pi]\}$ be a circle centered at $z_0$ with radius $r$. For fixed $m\in\mathbb{N}$, and a random $\vec{f}\in\mathbb{C}^{2n+1}$, we approximate the spectral projection  by 
	\begin{equation*}
		\mathcal{R}_m \vec{f}:=\dfrac{1}{2m}\sum_{j=0}^{2m-1}r{\rm e}^{{\rm i }\theta_j}\vec{x}_j,
	\end{equation*}
	where $\theta_j:=\pi j/m$ for $j=0, 1, \ldots, 2m-1$, and $\vec{x}_j$ are the solutions of the following linear system
	\begin{equation*}
		(r{\rm e}^{{\rm i }\theta_j}\vec{I}_{2n+1}-\vec{W}_{ n}(\kappa))\vec{x}_j=\vec{f}
	\end{equation*} 
	with $\vec{W}_{ n}(\kappa)=\vec{S}_n(\kappa)$ or $\vec{W}_{ n}(\kappa)=\vec{I}_{2n+1}+\vec{D}_n(\kappa)$. If there are no eigenvalues inside $\gamma$,  then we  have $\mathcal{R}_m\vec{f} \approx 0$. One defines an indicator function
	\begin{equation*}
		{\rm  RIM}_m(\kappa):=\left\|\mathcal{R}_m\left(\dfrac{\mathcal{R}_m\vec{f}}{\|\mathcal{R}_m\vec{f}\|}\right)\right\|
	\end{equation*}
	and uses it decide whether there is an eigenvalue inside $\gamma$ or not. The parallel SIM can effectively compute all the eigenvalues of $\vec{S}_n(\kappa)$ (or $\vec{I}_{2n+1}+\vec{D}_n(\kappa)$) in $\Theta \subset \mathbb C$. We refer the readers to \cite{XiSun2024arXiv} for details of implementation and codes. 
	
	We present three examples to verify the theory in Section 3 by computing the scattering poles in  $\Theta:=\set{x+{\rm i}y: x\in(0,4), y\in(-4,0)}$. The eigenvalues are approximated by {\bf pmCIMb} in \cite{XiSun2024arXiv}.
	
	\begin{example}\label{exa:disk}
		Let $\Omega$ be the unit disk. In this case, the scattering poles are the zeros of $H_\nu^{(1)}(\kappa), \nu \ge 2$, where  $H_\nu^{(1)}$ is the Hankel function of the first kind of order $\nu$ \cite{MaSun2023AML}. We plot the values of  $\log{\rm  RIM}_m(\kappa)$ for $\vec{S}_n(\kappa)$ and $\vec{I}_{2n+1}+\vec{D}_n(\kappa)$ with $n=32$ in $\Theta$ in Fig \ref{fig:disk1}, where the locations of the poles can be clearly seen. The result is consistent with that in \cite{MaSun2023AML}. The value of ${\rm  RIM}_m(\kappa)$ is close to zero except for the location near the scattering poles. We first compute the roots of $H_\nu^{(1)}$ to find three poles $\kappa_1=1.3080120323 - 1.6817888047i$, $\kappa_2=3.1130829450 -2.2186262746i$ and $\kappa_3=1.3038823977 -3.1351328447i$ with the accuracy of $10^{-10}$. The computed poles for different $n$'s are shown in Tables \ref{Disc1}-\ref{Disc3}. It can be seen that the computed poles converges quickly and achieves very highly accuracy with a relative small $n$. The values computed using  $\vec{S}_n(\kappa)$ and $\vec{I}_{2n+1}+\vec{D}_n(\kappa)$ agree with each other.
		
        Let $AE_S(n)$ and $AE_D(n)$ denote the absolute errors for  $\vec{S}_n(\kappa)$ and $\vec{I}_{2n+1}+\vec{D}_n(\kappa)$  with $n=5$, $6$, $7$, $8$, $9$ and $10$. We show $AE_S(n)$ and $AE_D(n)$ for the eigenvalues $\kappa_1$ and $\kappa_2$ in Fig \ref{fig:disk2}, which indicates that the convergence is exponential. 
	\end{example}

	\vspace{-0.5em}
	\begin{figure}[h!]
		\begin{center}
			\begin{tabular}{cc}
				\resizebox{0.5\textwidth}{!}{\includegraphics{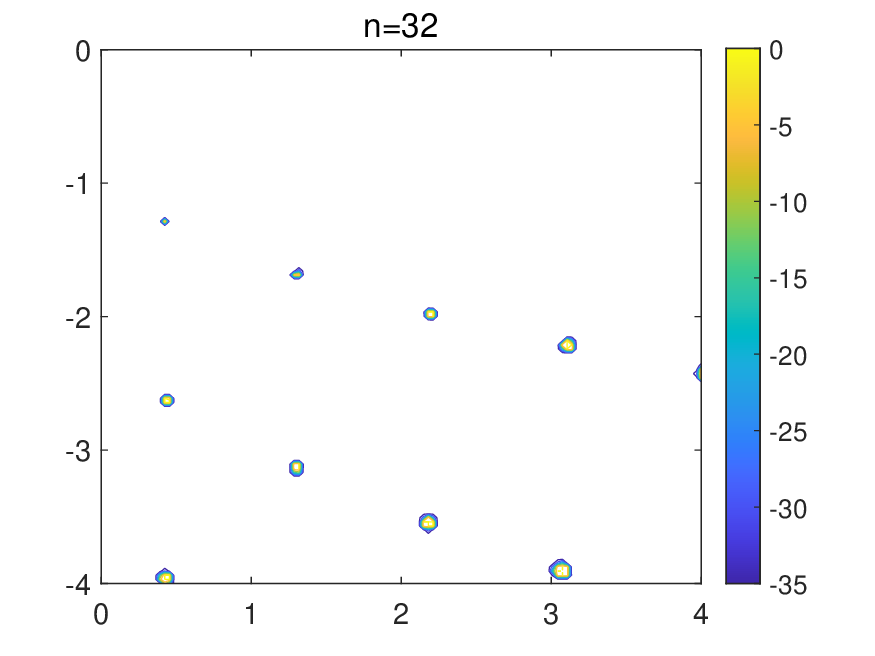}}&
				\resizebox{0.5\textwidth}{!}{\includegraphics{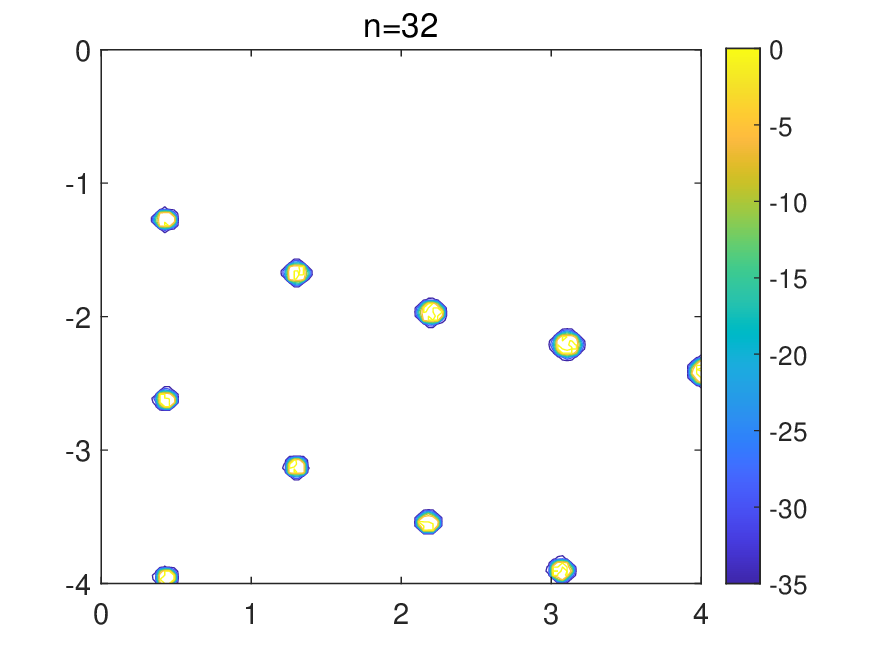}}
			\end{tabular}
		\end{center}
		\vspace{-1.5em}
		\caption{$\log{{\rm  RIM}_m(\kappa)}$ for  Example \ref{exa:disk} : Left: $\mathcal{S}_n(\kappa)$. Right: $\vec{I}_{2n+1}+\vec{D}_n(\kappa)$}
		\label{fig:disk1}
	\end{figure}

	\begin{center}
		\begin{table}[ht]
			\caption{Computed pole as the zero of $\vec{W}_n(\kappa)$ for Example \ref{exa:disk}. }
			\label{Disc1}
			\begin{center}
				\small
				\begin{tabular}{c|c|c}
					\hline
					n	&   $\vec{S}_n(\kappa)$ &  $\vec{I}_{2n+1}+\vec{D}_n(\kappa)$\\
					\hline
					$8$ & $1.308012032273757-1.681788804744781i$& $1.308012032273854 - 1.681788804742794i$\\
					\hline
					$16$ & $1.308012032273949 - 1.681788804745845i$& $1.308012032273949 - 1.681788804745844i$ \\
					\hline
					$32$ & $1.308012032273949 - 1.681788804745846i$&$1.308012032273949 - 1.681788804745845i$ \\
					\hline
					$64$ & $1.308012032273948 - 1.681788804745846i$& $1.308012032273948 - 1.681788804745844i$\\
					\hline
				\end{tabular}
			\end{center}
		\end{table}
	\end{center}
	
	\begin{center}
		\begin{table}[h!]
			\caption{Computed pole as the zero of $\vec{W}_n(\kappa)$ for Example \ref{exa:disk}. }
			\label{Disc2}
			\begin{center}
				\small
				\begin{tabular}{c|c|c}
					\hline
					n	&   $\vec{S}_n(\kappa)$ &  $\vec{I}_{2n+1}+\vec{D}_n(\kappa)$\\
					\hline
					$8$ & $3.113082969542856 - 2.218626154286283i$& $3.113083029499494 - 2.218626067886159i$\\
					\hline
					$16$ & $3.113082944985948 - 2.218626274639880i$& $ 3.113082944985947 - 2.218626274639878i$ \\
					\hline
					$32$ & $3.113082944985948 - 2.218626274639877i$&$3.113082944985949 - 2.218626274639875i$ \\
					\hline
					$64$ & $3.113082944985951 - 2.218626274639876i$& $3.113082944985946 - 2.218626274639875i$\\
					\hline
				\end{tabular}
			\end{center}
		\end{table}
	\end{center}
	
	\begin{center}
		\begin{table}[h!]
			\caption{Computed pole as the zero of $\vec{W}_n(\kappa)$ for Example \ref{exa:disk}. }
			\label{Disc3}
			\begin{center}
				\small
				\begin{tabular}{c|c|c}
					\hline
					n	&   $\vec{S}_n(\kappa)$ &  $\vec{I}_{2n+1}+\vec{D}_n(\kappa)$\\
					\hline
					$8$ & $1.303882375745608 - 3.135132844043817i$& $1.303882361925792 - 3.135132840998595i$\\
					\hline
					$16$ & $1.303882397713727 - 3.135132844704634i$& $1.303882397713714 - 3.135132844704639i$ \\
					\hline
					$32$ & $1.303882397713703 - 3.135132844704641i$&$1.303882397713705 - 3.135132844704644i$ \\
					\hline
					$64$ & $1.303882397713699 - 3.135132844704636i$& $1.303882397713704 - 3.135132844704642i$\\
					\hline
				\end{tabular}
			\end{center}
		\end{table}
	\end{center}
	
			\vspace{-0.5em}
	\begin{figure}[h!]
		\begin{center}
			\begin{tabular}{cc}
				\resizebox{0.5\textwidth}{!}{\includegraphics{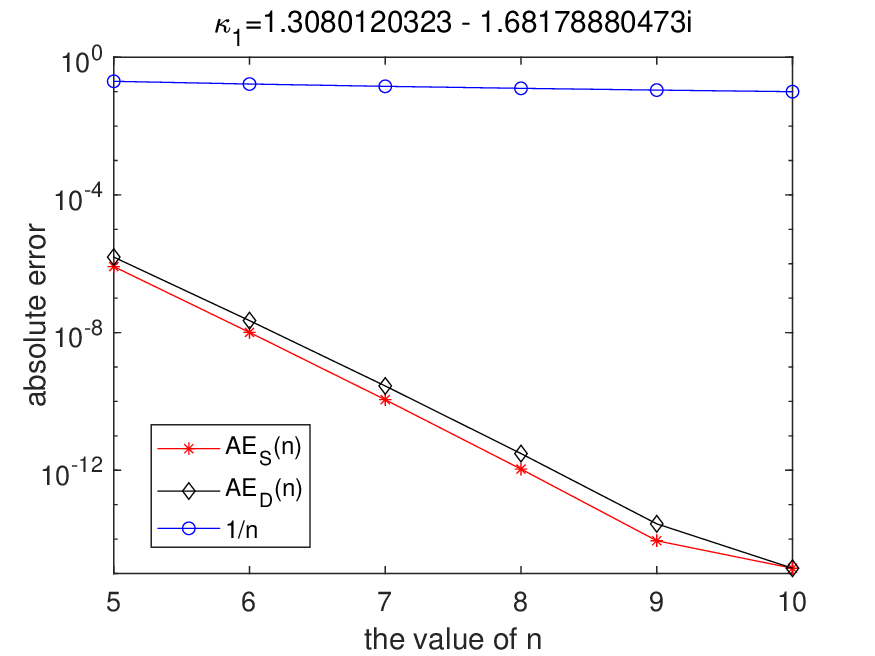}}&
				\resizebox{0.5\textwidth}{!}{\includegraphics{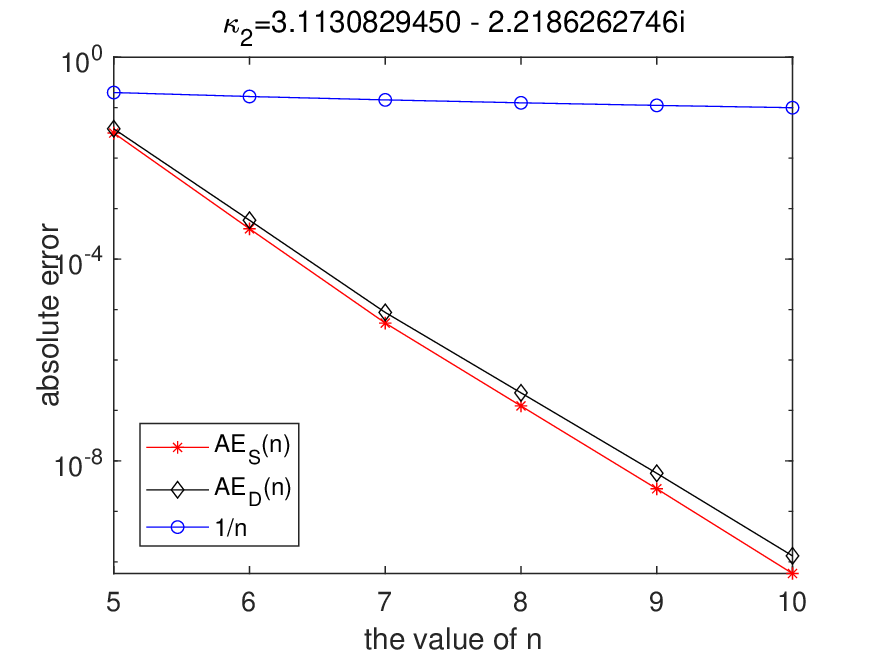}}
			\end{tabular}
		\end{center}
		\vspace{-1.5em}
		\caption{absolute error for  Example \ref{exa:disk} : Left: $\kappa_1$. Right: $\kappa_2$}
		\label{fig:disk2}
	\end{figure}

	\begin{example}\label{exa:peanut}
		We consider a peanut-shaped domain $\Omega$ whose boundary is given by 
		$$
		\sqrt{0.25+\cos^2 t}\,(\cos t, \sin t), \quad t\in[0, 2\pi). 
		$$
		We plot $\log {\rm  RIM}_m(\kappa)$ for  $\vec{S}_n(\kappa)$ and $\vec{I}_{2n+1}+\vec{D}_n(\kappa)$ with $n=32$ in Fig \ref{fig:peanut}. Computed zeros of $\vec{S}_n(\kappa)$ and   $\vec{I}_{2n+1}+\vec{D}_n(\kappa)$ for different $n$ are presented in Tables \ref{Peanut1}- \ref{Peanut2}. High accuracy is achieved with a relative small $n$. The values computed using  $\vec{S}_n(\kappa)$ and $\vec{I}_{2n+1}+\vec{D}_n(\kappa)$ agree with each other.
	\end{example}

	\begin{figure}[ht]
		\begin{center}
			\begin{tabular}{cc}
				\resizebox{0.5\textwidth}{!}{\includegraphics{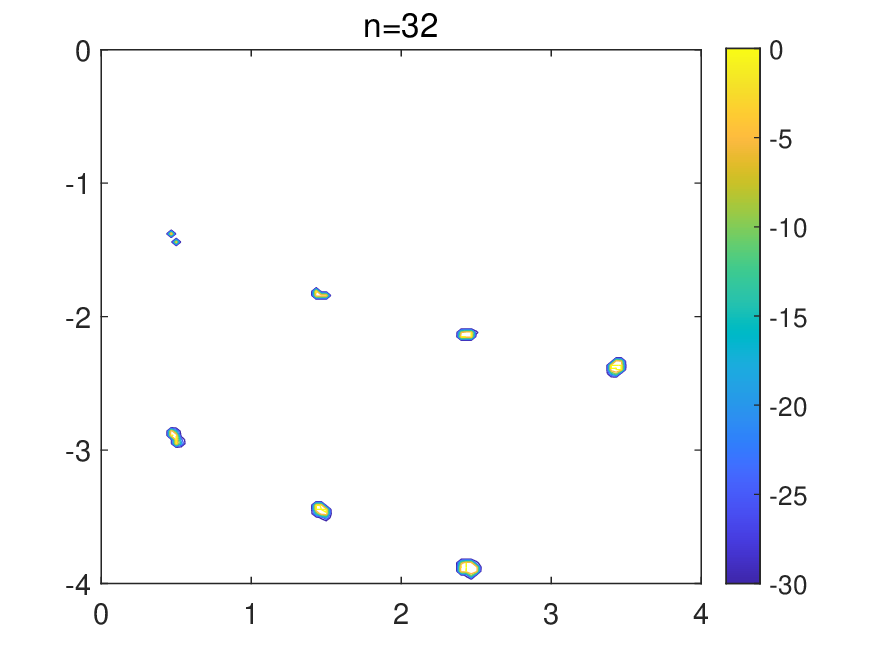}}&
				\resizebox{0.5\textwidth}{!}{\includegraphics{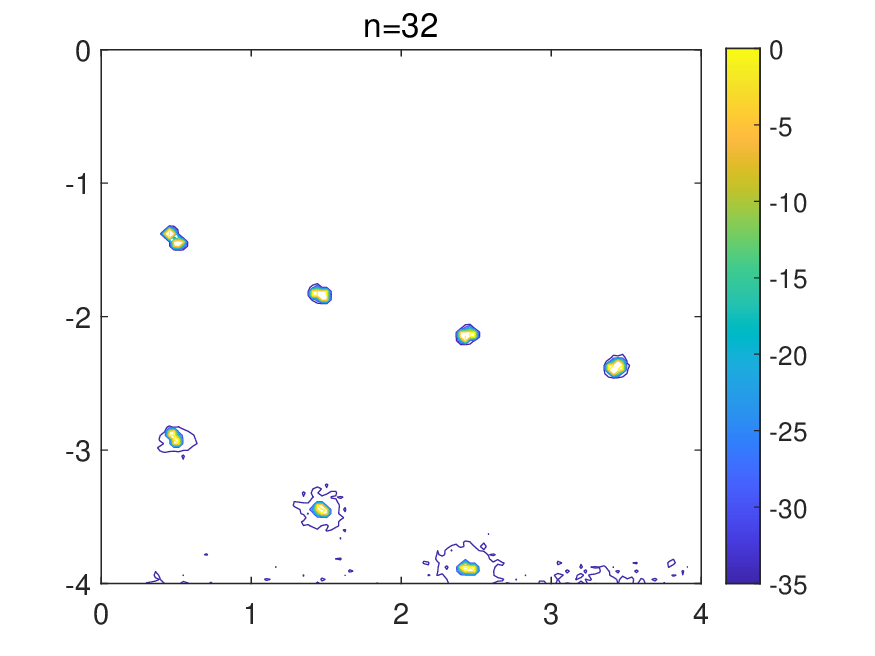}}
			\end{tabular}
		\end{center}
		\vspace{-1em}
		\caption{$\log{{\rm  RIM}_m(\kappa)}$ for  Example \ref{exa:peanut} : Left: $\mathcal{S}_n(\kappa)$. Right: $\vec{I}_{2n+1}+\vec{D}_n(\kappa)$}
		\label{fig:peanut}
	\end{figure}
	
	\begin{center}
		\begin{table}[h!]
			\caption{Computed pole as the zero of $\vec{W}_n(\kappa)$ for Example \ref{exa:peanut}. }
			\label{Peanut1}
			\begin{center}
				\small
				\begin{tabular}{c|c|c}
					\hline
					n	&   $\vec{S}_n(\kappa)$ &  $\vec{I}_{2n+1}+\vec{D}_n(\kappa)$\\
					\hline
					$8$ & $0.512610325138307 - 1.450154397792730i$& $0.512923981466455 - 1.450228641456581i$\\
					\hline
					$16$ & $0.513058892023026 - 1.450268225906149i$& $0.513059001432791 - 1.450268317679559i$ \\
					\hline
					$32$ & $0.513059002353327 - 1.450268319362658i$&$0.513059002368638 - 1.450268319377324i$ \\
					\hline
					$64$ & $0.513059002368638 - 1.450268319377325i$& $0.513059002368639 - 1.450268319377327i$\\
					\hline
				\end{tabular}
			\end{center}
		\end{table}
	\end{center}
	
	\begin{center}
		\begin{table}[h!]
			\caption{Computed pole as the zero of $\vec{W}_n(\kappa)$ for Example \ref{exa:peanut}. }
			\label{Peanut2}
			\begin{center}
				\small
				\begin{tabular}{c|c|c}
					\hline
					n	&   $\vec{S}_n(\kappa)$ &  $\vec{I}_{2n+1}+\vec{D}_n(\kappa)$\\
					\hline
					$8$ & $1.514816785260778 - 3.448351311881411i$& $1.479461533298764 - 3.435647225729820i$\\
					\hline
					$16$ & $1.450115980679582 - 3.441401566319377i$& $1.450573654621246 - 3.441041904050631i$ \\
					\hline
					$32$ & $1.450590990544579 - 3.441027020839657i$&$1.450590990128027 - 3.441027019823299i$ \\
					\hline
					$64$ & $1.450590990127992 - 3.441027019823282i$& $1.450590990128039 - 3.441027019823282i$\\
					\hline
				\end{tabular}
			\end{center}
		\end{table}
	\end{center}

	\begin{example}\label{exa:acorn}
		We consider an acorn-shaped domain $\Omega$ whose boundary is given by 
		$$
		0.6\sqrt{17/4+2\cos 3t}\,(\cos t, \sin t), \quad t\in[0, 2\pi). 
		$$
		In Fig \ref{fig:acorn}, we plot $\log{\rm  RIM}_m(\kappa)$ for  $\vec{S}_n(\kappa)$ and $\vec{I}_{2n+1}+\vec{D}_n(\kappa)$ with $n=32$ in $\Theta$. The computed zeros of $\vec{S}_n(\kappa)$ and   $\vec{I}_{2n+1}+\vec{D}_n(\kappa)$ for different $n$ are shown in Tables \ref{Acorn1}- \ref{Acorn2}. Similar to the previous two examples, high accuracy is achieved with a relative small $n$. Again, the values computed using  $\vec{S}_n(\kappa)$ and $\vec{I}_{2n+1}+\vec{D}_n(\kappa)$ agree with each other.
	\end{example}
	
	\begin{figure}[h!]
		\begin{center}
			\begin{tabular}{cc}
				\resizebox{0.5\textwidth}{!}{\includegraphics{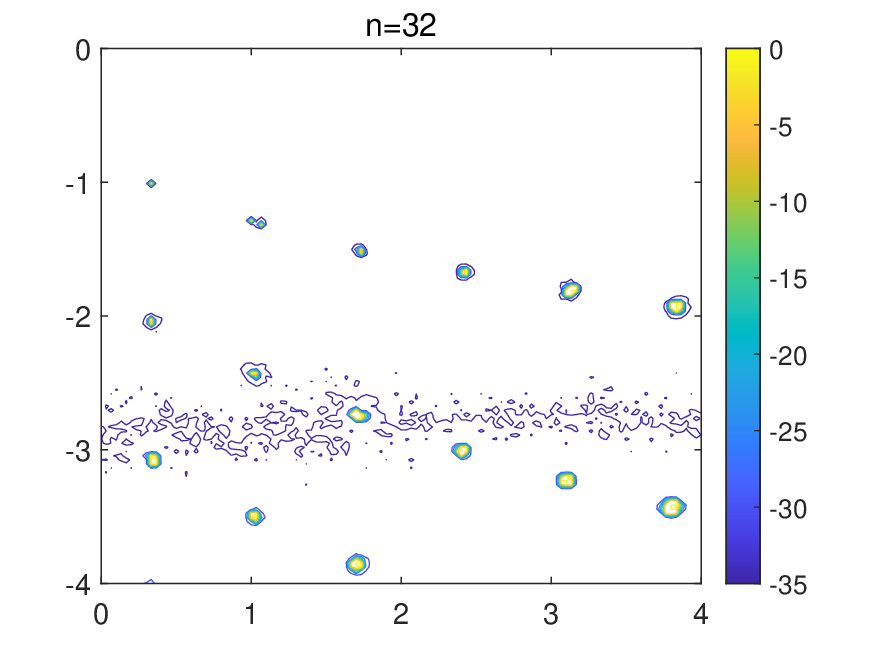}}&
				\resizebox{0.5\textwidth}{!}{\includegraphics{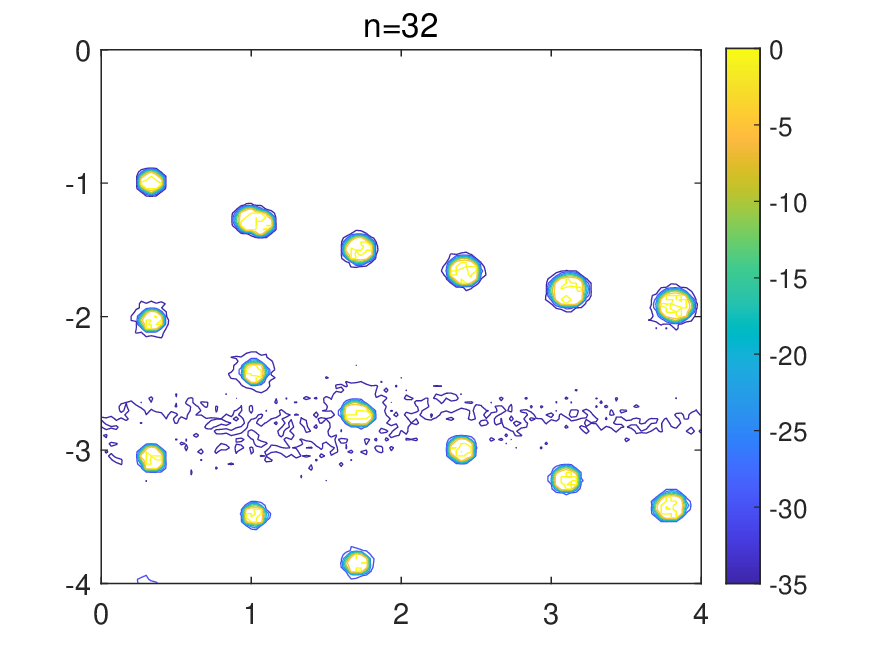}}
			\end{tabular}
		\end{center}
		\vspace{-1em}
		\caption{$\log{{\rm  RIM}_m(\kappa)}$ for  Example \ref{exa:acorn} : Left: $\mathcal{S}_n(\kappa)$. Right: $\vec{I}_{2n+1}+\vec{D}_n(\kappa)$}
		\label{fig:acorn}
	\end{figure}

	\begin{center}
		\begin{table}[h!]
			\caption{Computed pole as the zero of $\vec{W}_n(\kappa)$ for Example \ref{exa:acorn}. }
			\label{Acorn1}
			\begin{center}
				\small
				\begin{tabular}{c|c|c}
					\hline
					n	&   $\vec{S}_n(\kappa)$ &  $\vec{I}_{2n+1}+\vec{D}_n(\kappa)$\\
					\hline
					$8$ & $ 1.058802098401044 - 1.310170026426000i$& $1.064328363767797 - 1.309127243755672i$\\
					\hline
					$16$ & $1.064342851996245 - 1.309659237771693i$& $1.064343953354416 - 1.309657272685452i$ \\
					\hline
					$32$ & $1.064344075109297 - 1.309657355003215i$&$1.064344075189831 - 1.309657354813591i$ \\
					\hline
					$64$ & $1.064344075189833 - 1.309657354813590i$& $1.064344075189830 - 1.309657354813590i$\\
					\hline
				\end{tabular}
			\end{center}
		\end{table}
	\end{center}
	
	\begin{center}
		\begin{table}[h!]
			\caption{Computed pole as the zero of $\vec{W}_n(\kappa)$ for Example \ref{exa:acorn}. }
			\label{Acorn2}
			\begin{center}
				\small
				\begin{tabular}{c|c|c}
					\hline
					n	&   $\vec{S}_n(\kappa)$ &  $\vec{I}_{2n+1}+\vec{D}_n(\kappa)$\\
					\hline
					$8$ & $2.321023198349530 - 3.053291555439695i$& $2.302394469099680 - 3.179996502269232i$\\
					\hline
					$16$ & $2.409754984468690 - 3.007894134552222i$& $2.409737797779822 - 3.007661776632856i$ \\
					\hline
					$32$ & $2.409823640903252 - 3.007788190190519i$&$2.409822431695346 - 3.007787123777256i$ \\
					\hline
					$64$ & $2.409822431724737 - 3.007787123094735i$& $2.409822431724733 - 3.007787123094671i$\\
					\hline
				\end{tabular}
			\end{center}
		\end{table}
	\end{center}
	
	
	\section{Conclusions and Discussions}
	In this paper, we consider the computation of scattering poles for a sound-soft obstacle. These poles correspond to the eigenvalues of certain boundary integral operators. We show that they are Fredholm operators of index zero. Then we propose a Fourier-Galerkin method to discretize the integral operators and prove the convergence. We discuss the implementation of the method in detail. Numerical examples are provided to validate the theory and demonstrate the effectiveness of the proposed method.
	
	The use of boundary integral operators for computing scattering poles requires discretization only on the boundary of the obstacle, leading to a smaller algebraic system than other methods such as the finite element method. Additionally, the outgoing condition is naturally satisfied, and no spurious modes are introduced. When combined with the parallel spectral indicator method, the proposed approach offers an efficient and effective tool for computing scattering poles.
	
	In the future, we plan to extend the proposed method to compute scattering poles for inhomogeneous media. We are also interested in exploring scattering poles (resonances) for other wave phenomena, such as Maxwell's equations, Schr\"{o}dinger's equation, and periodic structures. In this paper, the boundary of the obstacle is assumed to be rather smooth. Effective methods for Lipschitz domains should be investigated.
	

	%
	
\end{document}